%% file: iclr2027_conference.tex
\documentclass{article} 
\usepackage{iclr2027_conference,times}

\input{math_commands.tex}

\usepackage{hyperref}
\usepackage{url}

\title{DCEmbed: Scalable Optimization over Neural Surrogates}

\author{Akshay Sreekumar, Ellen Vitercik \& Ram Rajagopal  \\
Stanford University\\
Stanford, CA, USA \\
\texttt{\{akshay81, vitercik, ramr\}@stanford.edu} \\
\And
Nicolas Christianson  \\
Johns Hopkins University \\
Baltimore, MD, USA \\
\texttt{christianson@jhu.edu} \\
\AND
Priya L. Donti  \\
Massachusetts Institute of Technology \\
Cambridge, MA, USA \\
\texttt{donti@mit.edu} \\
}

\iclrfinalcopy 

\newcommand{\wtilde}[1]{\widetilde{#1}}
\newcommand{\wbar}[1]{\overline{#1}}
\usepackage{amsmath}
\usepackage{amssymb}
\usepackage{amsthm}
\usepackage{booktabs}
\usepackage{thm-restate}
\newtheorem{theorem}{Theorem}[section]

\DeclareMathOperator{\ReLU}{ReLU}
\usepackage{algorithm}
\usepackage{algpseudocode}
\usepackage{listings}
\usepackage{graphicx}

\lstdefinestyle{paperpython}{
    language=Python,
    basicstyle=\ttfamily\footnotesize,
    keywordstyle=\bfseries,
    columns=fullflexible,
    keepspaces=true,
    showstringspaces=false,
    breaklines=true,
    captionpos=b
}

\usepackage{color-edits}
\addauthor{ev}{blue}

\begin{document}

\maketitle

\begin{abstract}
Neural surrogates can accelerate large-scale optimization by replacing expensive or intractable model components with more efficient learned approximations, but solving the resulting embedded problems can remain prohibitively costly. 
For instance, standard exact encodings of neural networks with $\ReLU$ activations allow for the embedded problem to be solved by mixed-integer solvers, but add large numbers of binary variables to accommodate the nonlinearity of the activations, which can render the problem computationally prohibitive.
To address this challenge, we propose \texttt{DCEmbed}, a local heuristic for optimization problems with embedded neural surrogates that leverages the difference-of-convex (DC) representation of the neural network and avoids introducing auxiliary activation binaries. 
Exploiting shared structure within the DC representation of a ReLU neural network, we derive a reduced-size, exact formulation for its convex components that can be embedded in optimization problems using just two linear inequalities and one continuous auxiliary variable per hidden neuron. 
Using this formulation, the resulting embedded problem can be solved via an iterative penalty convex-concave procedure, where only the concave portions of the neural terms are approximated at each stage. 
The original objective, constraints, and any discrete decisions are retained exactly, allowing the use of standard convex or mixed-integer optimization solvers to optimize the host and surrogate jointly at each iteration of the procedure. 
In experiments on synthetic quadratic programs, mixed-integer resource allocation, and neural two-stage stochastic programming, our method demonstrates substantially faster progress toward high-quality feasible solutions than approaches using exact mixed-integer embeddings. In particular, \texttt{DCEmbed} achieves $4\times$ lower normalized primal integral than the best exact baseline on the resource allocation problem, while in two-stage stochastic programming it reaches the global surrogate optimum $\sim 5\times$ faster than Gurobi ML.
\end{abstract}

\section{Introduction}
\label{sec:introduction}

Neural surrogates are increasingly used to approximate optimization model components that are expensive to evaluate or difficult to express analytically. 
Applications include approximating recourse costs in two-stage stochastic programming \citep{dumouchelle_neur2sp_2022}, modeling nonlinear physics in power system unit commitment \citep{kody_modeling_2022}, and incorporating constraints learned from data \citep{fajemisin_optimization_2024, maragno_mixed-integer_2025}.
We analyze settings in which trained neural networks replace selected components of the objective and/or constraints of an optimization problem that has discrete and/or continuous decision variables (the \emph{host model}). 
We refer to the remaining, explicitly modeled part as the \emph{residual model}. 
When the inputs to a neural surrogate depend on decision variables, the surrogate must be optimized jointly with this residual model. 
This can make the embedded problem difficult to solve even when mature optimization solvers can efficiently exploit the structure of the residual model. 

Standard exact mixed-integer encodings of $\ReLU$ networks retain the host objective and constraints, but introduce binary variables to represent neural activation states. 
The number of binaries scales with network size and repeated embeddings, leading to high-dimensional combinatorial problems whose solution can become computationally prohibitive. 
Alternative gradient-based solvers like projected gradient descent (PGD) avoid activation binaries by directly optimizing neural surrogates or penalties, but cannot directly accommodate discrete decisions in the host problem. 
Moreover, PGD can require costly projections onto the residual constraints, while poor numerical conditioning can slow progress. 
These limitations motivate a significant need for new methods that can use existing solvers to optimize the neural terms and residual model jointly while retaining any original discrete decisions and avoiding the combinatorial burden of neural activation binaries.


To address this gap, we propose \texttt{DCEmbed}, a scalable heuristic optimization method that accommodates neural surrogates in the objective as well as in inequality and equality constraints. 
We begin by expressing a trained $\ReLU$ network as a difference-of-convex (DC) functions \citep{awasthi_dc-programming_2024}.
At the core of \texttt{DCEmbed} is an exact, reduced-size formulation for optimizing over these convex components within the host model. 
We use this formulation within the penalty convex-concave procedure (CCP) \citep{lipp_variations_2016}, which solves a sequence of subproblems by retaining one convex component of each neural term and linearizing the other. 
Only the neural terms are approximated, with the residual model expressions and original binary decisions retained exactly. 
Off-the-shelf optimization solvers can therefore optimize the residual and convexified neural terms jointly, without the additional combinatorial search introduced by neural activation binaries.

\paragraph{Key Challenges.}
A central challenge is to turn the DC decomposition of a neural network into a representation that is both scalable and composable within a larger, structured optimization model. 
The formulation must recover the desired convex components exactly, keep the subproblems small as networks get larger or are embedded repeatedly, and reuse hidden variables when several predictions enter different parts of the model.
These requirements are complicated by the recursive dependence between the two convex components across layers in the DC decomposition.
The representation must also support neural objectives, inequalities, and equalities while preserving the residual expressions and original binary decisions. 
Additionally, CCP's conservative treatment of neural constraints can potentially exclude high quality feasible solutions and limit progress. 

\paragraph{Contributions.}
We make the following contributions.
\begin{itemize}
    \item \textbf{Exact neural formulations.}
    We show that one continuous auxiliary variable and two linear inequalities per hidden neuron are sufficient to represent the epigraphs of a trained $\ReLU$ network's DC components exactly, without activation binaries.
    Compared with a direct representation of both DC components, this halves the variable count and reduces the number of layerwise constraints required.
    The resulting formulation yields smaller CCP subproblems and supports the reuse of hidden variables and constraints.

    \item \textbf{Activation-binary-free optimization of embedded surrogates.}
    We combine this formulation with penalty CCP to optimize neural objectives, inequalities, and equalities jointly with the residual model and its original binary decisions.
    Penalty CCP enables exploration of points excluded by vanilla CCP's conservative approximation of neural constraints. 
    Because we add only continuous variables and linear inequalities, and no activation binaries, a convex residual model remains convex while a mixed-integer convex model remains mixed-integer convex.
    This allows existing mature optimization solvers to optimize the embedded surrogate and residual model together.

    \item \textbf{Superior computational scaling.}
    Across synthetic quadratic programs and two mixed-integer applications, \texttt{DCEmbed} finds strong feasible incumbents even when  exact mixed-integer formulations fail, and outperforms PGD on most tested continuous architectures. 
    On the largest resource allocation instance, it achieves over $4\times$ lower normalized primal integral than the best exact baseline. 
    In two-stage stochastic programming, it reaches the global surrogate optimum approximately $5\times$ faster than Gurobi ML at the largest scenario count, while exact mixed-integer formulations remain $4.81\%$ suboptimal at the time limit. 
    
\end{itemize}

\section{Related Work}
\label{sec:related_work}

\paragraph{Surrogate models in optimization.} 
Surrogate embeddings can vary in both the quantities learned and their role within the optimization model. 
Constraint learning augments an optimization model with feasibility conditions inferred from data \citep{fajemisin_optimization_2024, maragno_mixed-integer_2025}.
Other embeddings approximate physical relationships \citep{kody_modeling_2022}, or learn recourse functions and solution mappings of optimization problems \citep{dumouchelle_neur2sp_2022, chen_compact_2024}. 
Across these settings, the solution method depends on whether the learned terms enter the objective or constraints, and whether the host contains discrete decisions. 
Next, we discuss mixed-integer formulations, continuous optimization methods, and other alternatives for optimizing embedded models. 

\paragraph{Formulations for ReLU networks.}
The big-$M$ formulation is a standard exact encoding of a $\ReLU$ network using binary variables and linear constraints \citep{fischetti_deep_2018, grimstad_relu_2019}.
Using this encoding preserves the host objective and constraints but adds binary variables for the activations, turning continuous hosts into mixed-integer programs and increasing the number of discrete variables for 
mixed-integer hosts. 
Performance depends on the bounds on the preactivations, or inputs to each $\ReLU$, which determine the strength of the continuous relaxation. 
Bound tightening can strengthen these relaxations at the cost of additional preprocessing \citep{grimstad_relu_2019, sosnin_scaling_2024}, while stronger polyhedral formulations trade increased model size for tighter relaxations \citep{anderson_strong_2020, tsay_partition-based_2021}. 
Exact $\ReLU$ network embeddings are implemented in software packages such as OMLT \citep{ceccon_omlt_2022}, Gurobi Machine Learning \citep{gurobi_optimization_llc_gurobi_nodate}, and PySCIPOpt-ML \citep{turner_pyscipopt-ml_2024}. 
We instead optimize the embedded problem locally through a DC representation, preserving the host structure without introducing activation binaries or requiring additional computation for bound tightening. 

\paragraph{Continuous optimization over trained neural networks.}
For continuous hosts, projected gradient descent (PGD) \citep{beck_first-order_2017} differentiates through the fixed network and projects onto the residual constraints, with neural constraint violations handled through penalty terms. 
However, costly projections and poor numerical conditioning can limit the efficacy of PGD, and it does not directly accommodate discrete host decisions. 
\citet{awasthi_dc-programming_2024} construct a DC decomposition for $\ReLU$ networks and apply the difference-of-convex algorithm (DCA) to scalar neural objectives over convex domains, such as scalar function and adversarial margin minimization. 
In both applications, the network enters the host through a scalar objective optimized over a simple-to-project-onto convex input domain. 
\citet{liu_optimization_2025} propose a different DC formulation based on complementarity conditions for individual $\ReLU$s, whose violations are penalized before applying DCA. 
Here, we derive an exact, reduced-size epigraph formulation for the convex components in the network-level decomposition of \citet{awasthi_dc-programming_2024}.
A single set of continuous hidden variables and linear constraints serves to represent both convex components.
This shared formulation enables penalty CCP to optimize neural objectives, inequalities, and equalities jointly with the residual objective, constraints, and discrete decisions. 

\paragraph{Neural network verification and optimization.}
Neural network verifiers construct bounds on a network or computation graph over a bounded input region, refining the domain when the bounds are inconclusive. 
The same machinery can be used for continuous global optimization. 
In $\alpha,\beta$-CROWN, gradient-based search provides feasible incumbents, while a combination of certified bounds and branch-and-bound excludes input subdomains that cannot contain a better solution \citep{wang_beta-crown_2021, zhang_branch_2022, li_bridging_2026}.
In this case, both the host and neural terms enter the verification-based bounding framework. We instead retain the host objective and constraints explicitly within each CCP subproblem, allowing the underlying solver to handle their convex structure and any original discrete decisions directly.

\section{Problem Setting and Background}
We consider a \emph{host optimization model} with continuous decision variables $y\in \mathbb{R}^{n_y}$ and binary decision variables $\delta \in \{0,1\}^{n_\delta}$. The components to be replaced by neural surrogates are denoted $q_{\mathrm{obj}}$ in the objective, $q_{\mathrm{ineq},i}$ in
inequality constraint $i$, and $q_{\mathrm{eq},j}$ in equality
constraint $j$. For simplicity of notation, let $x$ denote the set of decision variables used as input to these functions. Thus, $x$ is a subvector of $(y,\delta)$.
Separating these components from the remaining model, we can write
\begin{equation}
\label{eq:host_problem}
\begin{aligned}
    \underset{y,\delta}{\operatorname{minimize}}\quad
        & \phi_0(y,\delta)+q_{\mathrm{obj}}(x)\\
    \text{s.t.}\quad
        & \phi_i(y,\delta)+q_{\mathrm{ineq},i}(x)\leq 0,
        && i=1,\ldots,m_{\mathrm{ineq}},\\
        & \psi_j(y,\delta)+q_{\mathrm{eq},j}(x)=0,
        && j=1,\ldots,m_{\mathrm{eq}},\\
        & y\in\mathcal{Y},\quad
          \delta\in\{0,1\}^{n_\delta},
\end{aligned}
\end{equation}
where $\mathcal{Y} \subseteq \mathbb{R}^{n_y}$ is a closed set. 
The scalar-valued $q$ terms denote possibly nonconvex, expensive to evaluate, or otherwise difficult components meant to be replaced by neural surrogates. The remaining expressions $\phi_0, \phi_i$, and $\psi_j$, together with the variables, constitute the \emph{residual model}. 


To approximate the $q$ terms in Problem~\ref{eq:host_problem}, we use a standard neural surrogate architecture~\citep{goodfellow_deep_2016} and take an affine combination of its outputs. In our notation, $f_{\theta} : \mathbb{R}^{n_0} \rightarrow \mathbb{R}^{p}$ is a network with fixed parameters $\theta$. 
For a network with $L$ $\ReLU$ layers, let $z^0(x) = x$ and define
\begin{equation}
\begin{aligned}
    a^\ell(x)
        &= W^\ell z^{\ell-1}(x) + b^\ell,
        && \ell=1,\ldots,L,\\
    z^\ell(x)
        &= \ReLU\left(a^\ell(x)\right),
        && \ell=1,\ldots,L,\\
    f_\theta(x)
        &= z^L(x),
\end{aligned}
\label{eq:relu_network}
\end{equation}
where $\operatorname{ReLU}(v)_k=\max\{v_k,0\}$ is applied componentwise, and $W^\ell\in\mathbb{R}^{n_\ell\times n_{\ell-1}}$ and $b^\ell\in\mathbb{R}^{n_\ell}$ for $\ell=1,\ldots,L$.
The approximation of each $q$ term in Problem~\ref{eq:host_problem} is specified by coefficient vectors $\nu \in \mathbb{R}^p$ and offsets $\beta \in \mathbb{R}$ (indexed by their role). In this way, the \emph{embedded problem} is
\begin{equation}
\label{eq:embedded_problem}
\begin{aligned}
    \underset{y,\delta}{\operatorname{minimize}}\quad
        & \phi_0(y,\delta)
          + \left(\nu_{\mathrm{obj}}^\top f_\theta(x)
          + \beta_{\mathrm{obj}}\right)\\
    \text{s.t.}\quad
        & \phi_i(y,\delta)
          + \left(\nu_{\mathrm{ineq},i}^\top f_\theta(x)
          + \beta_{\mathrm{ineq},i}\right)
          \leq 0,
        && i=1,\ldots,m_{\mathrm{ineq}},\\
        & \psi_j(y,\delta)
          + \left(\nu_{\mathrm{eq},j}^\top f_\theta(x)
          + \beta_{\mathrm{eq},j}\right)
          =0,
        && j=1,\ldots,m_{\mathrm{eq}},\\
        & y \in\mathcal{Y},\quad
          \delta\in\{0,1\}^{n_\delta}.
\end{aligned}
\end{equation}
For simplicity, Problem~\ref{eq:embedded_problem} shows a single network evaluation shared across the objective and constraints. 
If the objective or a constraint has no surrogate term, we set its corresponding coefficient vector $\nu$ and offset $\beta$ to be zero. Surrogate predictions can either enter directly in the constraints and objective as above, or an equality constraint can link a prediction to a host variable which can then appear elsewhere in the model. 
\subsection{Difference-of-Convex (DC) Representation}

A function $f$ is \emph{difference-of-convex (DC)} if it can be written as $f=g-h$, with $g$ and $h$ convex. 
The following result constructs a DC decomposition of each hidden activation in a ReLU network. To state the construction, define the positive and negative parts of a vector or matrix $M$ entrywise as $M^+=\max\{M,0\}$ and $M^-=\max\{-M,0\}$, so $M = M^+ - M^-$ and $|M| = M^+ + M^-$.

\begin{restatable}[\citealp{awasthi_dc-programming_2024}]{theorem}{dcdecomposition}
\label{thm:dc_decomposition}
For each hidden layer $\ell=1,\ldots,L$ of a $\ReLU$ network, there exist componentwise convex functions $G^\ell$ and $H^\ell$ such that $z^\ell(x)=G^\ell(x)-H^\ell(x)$.
One such decomposition is given by the following recursion. Suppressing the argument $x$ for simplicity, initialize $G^1=z^1, H^1=0$ for $\ell=1$, and for $\ell \geq 2$, set
\begin{equation}
\begin{aligned}
    &H^\ell=W^{\ell,-}G^{\ell-1} + W^{\ell,+}H^{\ell-1}, \quad
    &G^\ell = \max\left\{W^{\ell,+}G^{\ell-1} + W^{\ell,-}H^{\ell-1} + b^\ell, H^\ell\right\},
\end{aligned}
\label{eq:awasthi_relu_recurrence}
\end{equation}
where the maximum is taken componentwise.
\end{restatable}

We include a proof of this result in Appendix~\ref{appendix:dc_decomp_proofs} for completeness. 
This construction provides recursive expressions for the convex components within the DC decomposition, but does not give an explicit formulation of their epigraphs using continuous variables and linear constraints. 
As we will show in Section~\ref{section:method}, obtaining such a formulation will allow the components to be optimized jointly with the residual model while retaining its structure and avoiding neural activation binaries. 
\subsection{Convex--Concave Procedure (CCP)}
The convex-concave procedure (CCP) constructs convex upper approximations of DC functions by retaining one convex component and linearizing the other \citep{lipp_variations_2016}. 
For a scalar DC function $f=g-h$, we define the affine lower bound of $h$ at iterate $x^k$ as
\begin{equation}
    \widehat h^k(x) = h(x^k) + (d_h^k)^\top (x-x^k), \quad d_h^k \in \partial h(x^k),
\end{equation}
and define $\widehat g^k(x)$ analogously. 
Since $h$ and $g$ are convex, we have $\widehat h^k(x) \leq h(x)$ and $\widehat g^k(x) \leq g(x)$. Thus, 
\begin{equation}
    f(x) \leq g(x) - \widehat h^k(x), \quad  -f(x) \leq h(x) - \widehat g^k(x).
\end{equation}
The first bound provides an upper approximation when $f$ appears in the (minimization) objective, and a convex restriction when $f$ appears in an inequality constraint. 
Equalities are treated by applying the restriction in both directions. CCP solves the resulting subproblem and updates the affine approximations at the new iterate, repeating until a stopping criterion is met.
Noting that these restrictions can exclude points that satisfy the original constraints, penalty CCP relaxes the restrictions with nonnegative slack variables and penalizes their sum in the objective, allowing exploration of otherwise infeasible points in the restricted set. The next section will discuss how this procedure can be efficiently employed for Problem~\ref{eq:embedded_problem} under our exact formulation of the neural DC components.

\section{Method} \label{section:method}
In this section, we develop an exact formulation of the neural DC components that reduces the size of each CCP subproblem. 
Because these subproblems are solved repeatedly, reducing their size can significantly lower the overall computational cost of CCP, particularly for deeper networks and repeated surrogate embeddings. 
First, we show in Theorem~\ref{thm:shared_components} that both components can be expressed directly in terms of the network's hidden activations.
Using these expressions, we establish an exact epigraph formulation with one continuous auxiliary variable and two linear inequalities per hidden neuron in Theorem~\ref{thm:extended_representation}.
We use this representation within penalty CCP to optimize the neural objectives and constraints jointly with the residual model. 
\label{sec:method}
\begin{theorem}[Shared representation of DC components]
    \label{thm:shared_components}
    For each affine measurement in Problem~\ref{eq:embedded_problem}, let $f_r(x) = \nu_r^\top z^L(x) + \beta_r$, where $r$ denotes $\mathrm{obj}$, $(\mathrm{ineq},i)$, or  $(\mathrm{eq},j)$. We define:
    \begin{equation}
    \begin{aligned}
        \rho_r^L
          &= |\nu_r|, \\
        \rho_r^{\ell-1}
          &= |W^\ell|^\top \rho_r^\ell,
          && \ell = L, L-1, \ldots, 2, \\
        c_{r,\ell-1}
          &= \bigl(W^{\ell,-}\bigr)^\top \rho_r^\ell,
          && \ell = L, L-1, \ldots, 2.
    \end{aligned}
\label{eq:skip_coefficient_recurrence}
\end{equation}
Then, $f_r = G_r - H_r$ has the convex components
\begin{equation}
\label{eq:shared_components}
\begin{aligned}
    G_r(x)
        =\beta_r+(\nu_r^+)^\top z^L(x)
          +\sum_{\ell=1}^{L-1}c_{r,\ell}^\top z^\ell(x),\quad
    H_r(x)
        =(\nu_r^-)^\top z^L(x)
          +\sum_{\ell=1}^{L-1}c_{r,\ell}^\top z^\ell(x).
\end{aligned}
\end{equation}
\end{theorem}
\begin{proof}
    Applying Theorem~\ref{thm:dc_decomposition} to the last hidden layer gives $z^L = G^L - H^L$. Writing $\nu_r=\nu_r^+-\nu_r^-$, we have the affine measurement representation $f_r=G_r-H_r$ with
    \begin{equation}
    \label{eq:measurement_dc}
    \begin{aligned}
        G_r
            =\beta_r+(\nu_r^+)^\top G^L+(\nu_r^-)^\top H^L, \quad
        H_r
            =(\nu_r^-)^\top G^L+(\nu_r^+)^\top H^L.
    \end{aligned}
    \end{equation}
    We define the parameterization $\eta^\ell=H^\ell$, and thus $G^\ell=z^\ell+\eta^\ell$.
    Substituting into Eq.~\ref{eq:awasthi_relu_recurrence} gives
    \begin{equation}
    \label{eq:eta_recurrence}
        \eta^1=0,
        \qquad
        \eta^\ell
        =W^{\ell,-}z^{\ell-1}+|W^\ell|\eta^{\ell-1},
        \quad \ell=2,\ldots,L.
\end{equation}
Substituting the parameterization into Eq.~\ref{eq:measurement_dc} gives
\begin{equation}
\begin{aligned}
G_r(x)=\beta_r+\left(\nu_r^+\right)^\top z^L(x)+ \left|\nu_r\right|^\top\eta^L(x), \quad
H_r(x)=\left(\nu_r^-\right)^\top z^L(x)+\left|\nu_r\right|^\top\eta^L(x).
\end{aligned}
\label{eq:scalar_components_with_eta}
\end{equation}
Expanding~(\ref{eq:eta_recurrence}) and combining with $\eta^1=0$ yields the equivalence $ \left|\nu_r\right|^\top\eta^L(x)
    =
    \sum_{\ell=1}^{L-1}
    c_{r,\ell}^\top z^\ell(x)$.
Finally, substituting this equivalence into Eq.~\ref{eq:scalar_components_with_eta} yields the desired representation. 
\end{proof}
The following theorem gives exact epigraph formulations for the components in Eq.~\ref{eq:shared_components}
by replacing the hidden activations with continuous auxiliary variables and imposing a set of linear inequalities.

\begin{restatable}[Extended Representation of DC Components]{theorem}{extendedrepresentation}
\label{thm:extended_representation}
For any input $x$, set $u^0=x$ and define the set
\begin{equation}
\mathcal{U}(x)
=
\left\{
u=(u^1,\ldots,u^L)
\;\middle|\;
\begin{aligned}
u^\ell &\ge W^\ell u^{\ell-1}+b^\ell,\\
u^\ell &\ge 0,
\end{aligned}
\quad \ell=1,\ldots,L
\right\},
\label{eq:ux}
\end{equation}
where all inequalities are componentwise. 
Let $\wbar{G}_r(u)$ and $\wbar{H}_r(u)$ be the affine expressions obtained from Eq.~\ref{eq:shared_components} by replacing each $z^\ell(x)$ with $u^\ell$.
Then, for every input $x$,
\begin{equation}
G_r(x)
=
\min_{u\in\mathcal{U}(x)}\wbar{G}_r(u),
\qquad
H_r(x)
=
\min_{u\in\mathcal{U}(x)}\wbar{H}_r(u).
\end{equation}
\end{restatable}

\begin{proof}[Proof Sketch]
We fix $x$ and take any $u \in \mathcal{U}(x)$. We define the common offsets
\begin{equation*}
    \begin{aligned}
        &\wtilde{\eta}^1 = 0, \quad
        &\wtilde{\eta}^\ell = W^{\ell,-}u^{\ell-1} + |W^\ell|\wtilde{\eta}^{\ell-1}, \quad \ell=2,\dots,L.
    \end{aligned}
\end{equation*}
and
\begin{equation*}
    \wtilde{G}^{\ell}(x)=u^{\ell} + \wtilde{\eta}^{\ell}(x), \quad 
    \wtilde{H}^{\ell}(x)= \wtilde{\eta}^{\ell}(x).
\end{equation*}
We can show by induction that $\wtilde{G}^{\ell}(x) \geq G^{\ell}(x)$ and $\wtilde{H}^{\ell}(x) \geq H^{\ell}(x)$, which combined with the affine measurement representation ultimately gives $\wbar{G}_r(u) \geq G_r(x)$ and $\wbar{H}_r(u) \geq H_r(x)$.
Conversely, we can show that there exists a feasible $u$ that attains this bound by considering the choice $ u^\ell = z^\ell(x) = \ReLU \left(W^\ell z^{\ell-1}(x) + b^\ell \right)$. See Appendix~\ref{app:representation_proof} for the full proof. 
\end{proof}

\paragraph{Sharing and formulation size.}
The set $\mathcal{U}(x)$ (\ref{eq:ux}) depends only on the trained network and its input, while the measurement coefficients enter through the expressions $G_r(u)$ and $H_r(u)$ in Eq.~\ref{eq:shared_components}.
The assignment $u^{\ell}=z^{\ell}(x)$ for $\ell = 1, \dots, L$ simultaneously minimizes both expressions for all $r$. Therefore, the epigraph constructions of these convex components can share $u$ and its layerwise constraints across the objective and constraints of each CCP subproblem.
Let $N_{\mathrm{hidden}} = \sum_{\ell=1}^{L}n_\ell$ denote the total number of hidden units in an $L$ layer $\ReLU$ network.
Our formulation uses $N_{\mathrm{hidden}}$ continuous hidden variables and $2N_{\mathrm{hidden}}$ layerwise inequalities, compared with $2N_{\mathrm{hidden}} -n_1$ variables and $3N_{\mathrm{hidden}} -n_1$ constraints for a direct representation.
This reduction can substantially improve computational performance as demonstrated in Section~\ref{sec:syn_qps}.

\subsection{Penalty CCP for the Embedded Problem}
\label{sec:ccp}
We let $x^k$ denote the input selected from $(y^k, \delta^k)$ at iteration $k$. 
For each affine measurement $r$, let $\widehat H_r^k$ and $\widehat G_r^k$ denote the affine lower bounds of $H_r$ and $G_r$, respectively, at $x^k$.
 Applying Theorem~\ref{thm:extended_representation} to the embedded problem in (\ref{eq:embedded_problem}), the CCP subproblem at iteration $k$ becomes
\begin{equation}
\label{eq:embedded_ccp_subproblem}
\begin{aligned}
    \underset{y,\delta,u,s,t}{\operatorname{minimize}}\quad
        & \phi_0(y,\delta)
          +\wbar{G}_{\mathrm{obj}}(u)
          -\widehat H_{\mathrm{obj}}^k(x)
        +\tau_k\left(\sum_i s_i+\sum_j t_j\right)\\
    \text{s.t.}\quad
        & \phi_i(y,\delta)
          +\wbar{G}_{\mathrm{ineq},i}(u)
          -\widehat H_{\mathrm{ineq},i}^k(x)\leq s_i,
        &&i=1,\ldots,m_{\mathrm{ineq}},\\
        & \psi_j(y,\delta)
          +\wbar{G}_{\mathrm{eq},j}(u)
          -\widehat H_{\mathrm{eq},j}^k(x)\leq t_j,
        &&j=1,\ldots,m_{\mathrm{eq}},\\
        & -\psi_j(y,\delta)
          +\wbar{H}_{\mathrm{eq},j}(u)
          -\widehat G_{\mathrm{eq},j}^k(x)\leq t_j,
        &&j=1,\ldots,m_{\mathrm{eq}},\\
        & u\in\mathcal{U}(x),\quad s\geq0,\quad t\geq0, \quad y\in\mathcal{Y},\quad
          \delta\in\{0,1\}^{n_\delta},
\end{aligned}
\end{equation}
where $\tau_k > 0$ is the slack penalty parameter. 
For constraints without a neural term, the corresponding slack is fixed to zero. 
Each equality constraint has been expanded into two restrictions sharing the same slack $t_j$. 
Theorem~\ref{thm:extended_representation} ensures that sharing $u$ across the objective and constraints introduces no additional approximation of the penalized CCP subproblem beyond the neural linearizations. 

Problem~\ref{eq:embedded_ccp_subproblem} retains the residual expressions and original binary variables, adding only continuous auxiliary variables and linear inequalities for the convexified neural terms.
As a result, the subproblems above are convex or mixed-integer convex when $\mathcal{Y}$ is convex, $\phi_0$ and $\phi_i$ are convex on $\mathcal{Y}\times [0,1]^{n_\delta}$, and each $\psi_j$ is affine, even when the components $q$ being approximated are nonconvex. 

\paragraph{Neural feasibility and stopping criteria.} 
Let $w=(y, \delta)$ denote the original host variables. We evaluate the embedded objective and violation as
\begin{equation}
\label{eq:objective_and_violation}
\begin{aligned}
    F(w)
        &=\phi_0(y,\delta)+f_{\mathrm{obj}}(x),\\
    V(w)
        &=\max\left\{
            0,\;
            \max_i[\phi_i(y,\delta)+f_{\mathrm{ineq},i}(x)]_+,\;
            \max_j|\psi_j(y,\delta)+f_{\mathrm{eq},j}(x)|
          \right\},
\end{aligned}
\end{equation}
where $[a]_+ = \max\{a,0\}$. A candidate point is considered feasible when we have $V(w) \leq \epsilon_{\mathrm{feas}}$, where $\epsilon_{\mathrm{feas}}$ is a desired neural feasibility tolerance.
We terminate the algorithm when two consecutive feasible candidates satisfy $    \left |F(w^{k+1}) - F(w^k) \right | \leq \epsilon_{\mathrm{obj}} \max\{1, \left |F(w^k)\right|\}$.

\paragraph{Penalty update.}
Let $\mu > 1$ denote the penalty scaling factor and $\tau_{\mathrm{max}}$ the maximum penalty. Then, we set $\tau_{k+1}=\min\{\mu\tau_k,\tau_{\max}\}$ when $V(w^{k+1})>\epsilon_{\mathrm{feas}}$ and keep $\tau_{k+1}=\tau_k$ otherwise.

\section{Experiments}
\label{sec:experiments}
To assess \texttt{DCEmbed}'s performance, we evaluate how quickly it finds high quality feasible solutions within a fixed-time budget.
A standard metric to capture the tradeoff between speed and solution quality is the \emph{primal integral} \citep{berthold_measuring_2013, gasse_machine_2022}.
We use a \emph{normalized primal integral} (NPI) that scales the objective gap by the improvement available from a common initialization, making scores comparable across instances with different objective scales and initial gaps:
\begin{equation}
\label{eq:npi}
    \operatorname{NPI}
    = \frac{100}{T}\int_0^T
      \frac{J_{\mathrm{best}}(t)-J_{\mathrm{ref}}}
           {J_0-J_{\mathrm{ref}}}\,dt,
\end{equation}
where $J_{\mathrm{best}}(t)$ is the best feasible objective found by time $t$, $J_0$ is the initial objective, and $J_{\mathrm{ref}}$ is the best known feasible objective for the instance. 
Lower values indicate stronger primal progress, and $100\%$ indicates no improvement over the initial incumbent. 

All experiments use neural surrogates trained in PyTorch \citep{paszke_pytorch_2019}, with all optimization models formulated in Gurobi \citep{gurobi_optimization_llc_gurobi_2026}.
All baseline methods that use Gurobi are configured to prioritize finding feasible incumbents rather than solving to optimality. 
For detailed descriptions of the baselines and experiment settings, see Appendices~\ref{app:baseline_methods} and \ref{app:exp_details}.

\subsection{Network Scaling in Synthetic Convex Programs}
\label{sec:syn_qps}

We first study how \texttt{DCEmbed} scales with the width and depth of the embedded neural surrogate using synthetic quadratic programs (QPs) with repeated neural inequality constraints. 
We randomly generate convex QPs, where the decision vector $x\in \mathbb{R}^{150}$ is partitioned into $30$ blocks $x^{(r)} \in \mathbb{R}^5$:
\begin{equation}
\label{eq:random_qp}
\begin{aligned}
    \underset{x\in[-1,1]}{\operatorname{minimize}}\quad
        &\frac12\|Gx\|_2^2+c^\top x\\
    \text{s.t.}\quad
        &Ex=0,\quad |Cx|\leq\mu\mathbf{1},\\
        &f_\theta(x^{(r)})\leq0,\quad r=1,\ldots,30.
\end{aligned}
\end{equation} 
We train networks to approximate a fixed nonconvex constraint function and vary the width $w \in \{8,16,32,64,128,256\}$ at depth $d=3$, and vary the depth $d \in \{1,\ldots,6\}$ at width $w=64$.
We compare \texttt{DCEmbed}, Gurobi ML, big-$M$ with optimization-based bound tightening (OBBT)~\citep{grimstad_relu_2019}, and projected gradient descent (PGD). 
We repeat both sweeps over three independent trials using separately sampled residual models and newly trained networks. 
Within each trial, the residual model is fixed across architectures, and all methods receive the same initialization for each architecture. 
\begin{figure}[t]
    \centering
    \includegraphics[width=\linewidth]{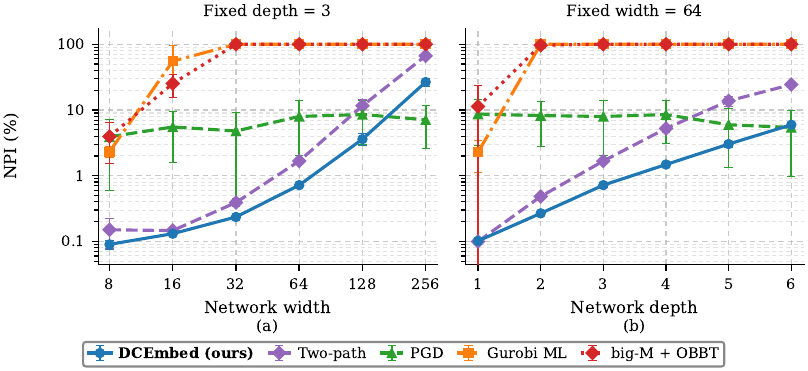}
    \caption{NPI (lower is better) on QPs as network (a) width and (b) depth increases. \texttt{DCEmbed} substantially outperforms the exact embeddings and improves on the two-path formulation.}
    \label{fig:random_qp_scaling}
\end{figure}
\paragraph{Results.}
Figure~\ref{fig:random_qp_scaling} shows that \texttt{DCEmbed} achieves substantially lower normalized primal integral than both Gurobi ML and big-$M$ with OBBT. 
Neither of these latter exact formulations improve upon the initial feasible incumbent for wider and deeper networks. 
At $(w,d)=(64,3)$, \texttt{DCEmbed} achieves a mean NPI of $0.72\%$ compared to $7.97\%$ for PGD and $100\%$ for both exact embeddings. 
\texttt{DCEmbed} achieves lower NPI than PGD across most architectures, although its advantage narrows as the neural network size grows, with PGD performing best at the largest width and slightly better at the largest depth. 
We also compare \texttt{DCEmbed}'s shared hidden representation with separate layerwise representations of the two convex components (\emph{two-path}).
The shared formulation reduces mean NPI by approximately $2.5$--$3.2\times$ at widths 128--256 and $3.5$--$4.5\times$ at depths 4--6. 

\subsection{Resource Allocation for Water Treatment}
We adapt the shared resource allocation benchmark from SurrogateLIB \citep{turner_pyscipopt-ml_2024} to evaluate repeated neural equalities within a mixed-integer residual model. 
The benchmark considers the allocation of shared resources to treat a set of water samples.
Each sample $i \in \mathcal{I}=\{1,\ldots,N\}$ has nine measured attributes. 
Let $b_i \in \mathbb{R}^9$ denote a sample's untreated feature vector, and let $r_i,a_i \in \mathbb{R}^9$ denote the amounts removed from and added to each attribute, producing the treated feature vector $x_i$.
The objective maximizes the number of samples predicted potable, subject to shared removal and addition budgets $R,A \in \mathbb{R}^9$.
The resulting embedded optimization problem is
\begin{equation}
\label{eq:water_problem}
\begin{aligned}
    \underset{x,r,a,s,y}{\operatorname{maximize}}\quad
        & \sum_{i\in\mathcal{I}} y_i\\
    \text{s.t.}\quad
        & x_i=b_i-r_i+a_i,\quad s_i=f_\theta(x_i),
        && i\in\mathcal{I},\\
        & y_i\in\{0,1\},\quad
          y_i=1\Longrightarrow s_i\geq\epsilon,
        && i\in\mathcal{I},\\
        & \sum_{i\in\mathcal{I}}r_i\leq R,\quad
          \sum_{i\in\mathcal{I}}a_i\leq A,\quad
          r,a\geq0.
\end{aligned}
\end{equation}
The binary variable $y_i$ allows a sample to be counted only when its classifier score $s_i$ meets the potability threshold $\epsilon$.
Each sample contributes one neural output equality and one original binary decision, with treatment decisions coupled through the shared budgets. 
The explicit score $s_i$ connects the surrogate to Gurobi's native indicator constraint, which remains part of the residual model. 

We use a fixed classifier trained on the Water Potability dataset \citep{kadiwal_water_2021} with four hidden $\ReLU$ layers of width 32 to predict sample potability.
We vary $N \in \{25,50,100,200,400,800\}$ over three candidate cohorts and compare \texttt{DCEmbed}, Gurobi ML, and big-$M$ with OBBT using the same network and initialization.
\begin{figure}[t]
    \centering
    \includegraphics[width=\linewidth]{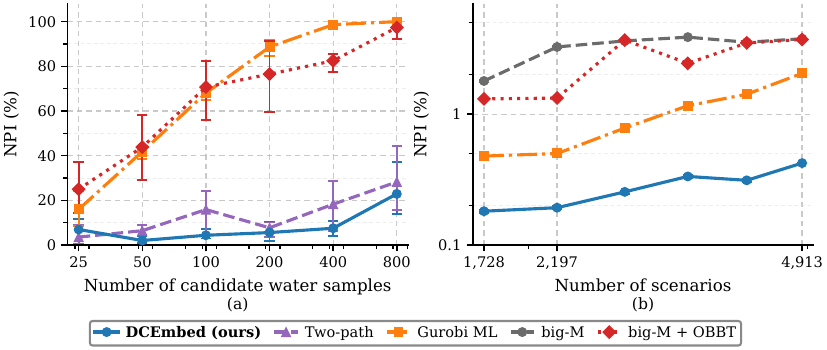}
    \caption{Scaling in (a) resource allocation for water treatment and (b) neural two-stage stochastic programming. \texttt{DCEmbed} achieves lower NPI than exact embeddings across all tested sizes.}
    \label{fig:neur2sp_water_scaling}
\end{figure}
\paragraph{Results.}
In Figure~\ref{fig:neur2sp_water_scaling}(a), we see \texttt{DCEmbed} achieves lower mean NPI against exact methods at every candidate count.
At $N=800$, it achieves an NPI of $22.6\%$ compared with $97.4\%$ for big-$M$ with OBBT and $100\%$ for Gurobi ML.
At $N=800$, the OBBT big-$M$ model retains roughly $89,160$ neural activation binaries in addition to the 800 original binary host variables. \texttt{DCEmbed} removes the neural activation binaries and solves a sequence of mixed-integer convex subproblems containing the original decisions. Compared to two-path, \texttt{DCEmbed} achieves lower mean NPI at all but the smallest size, with an average mean NPI improvement of $\sim2\times$ across all sizes. 
The smaller gains relative to the QP benchmark may reflect the additional difficulty of discrete subproblems. 

\subsection{Scenario Scaling in Two-Stage Stochastic Programming}
\label{sec:exp_neur2sp}

We revisit the pooling benchmark from Neur2SP \citep {dumouchelle_neur2sp_2022}, a two-stage stochastic program in which 16 binary first-stage decisions $x$ are chosen before uncertainty $\xi$ is realized. 
For each realization, a nonconvex second-stage problem determines the optimal recourse cost $Q(x,\xi)$. 
Neur2SP's per-scenario neural surrogate (\emph{NN-P}) learns $\Phi_\theta^P(x,\xi) \approx Q(x,\xi)$, replacing each second-stage optimization with a neural prediction. 
For $S$ scenarios with probabilities $p_s$, the embedded problem minimizes the first-stage cost plus the predicted expected recourse cost as
\begin{equation}
    \label{eq:neur2sp_nnp}
    \underset{x\in\mathcal{X}}{\operatorname{minimize}}
    \quad
    c^\top x+
    \sum_{s=1}^{S}p_s\Phi_\theta^{P}(x,\xi_s),
\end{equation}
where $\mathcal{X}$ contains the first-stage constraints and binary decisions. 
The objective embeds $S$ evaluations of the same trained neural network, all coupled through the first-stage decisions $x$.

We use the released \emph{NN-P} network with one hidden $\ReLU$ layer of width 64 and vary $S$ from $1,728$ to $4,913$ over 6 sizes. 
We compare \texttt{DCEmbed}, Gurobi ML, and Neur2SP's original and OBBT big-$M$ formulations on the same first-stage model, network, and scenarios across three trials.
\paragraph{Results.}
Figure~\ref{fig:neur2sp_water_scaling}(b) shows that \texttt{DCEmbed} achieves the lowest average NPI across all scenario counts under the $180s$ budget. At $S=4,913$, its NPI is $0.423\%$, approximately $4.8\times$ and $8.8\times$ lower than Gurobi ML and big-$M$ with OBBT, respectively. 
Both \texttt{DCEmbed} and Gurobi ML recover the global optimum (computed offline) of the neural surrogate problem, but \texttt{DCEmbed} reaches it approximately $5\times$ faster than Gurobi ML at $S=4,913$ scenarios, with median times of 18.82 and 95.54 seconds, respectively. 
At this size, both big-$M$ variants remain $4.81\%$ suboptimal at the time limit. 
Even after bound tightening, the OBBT formulation contains $170,714$ neural activation binaries. 
\texttt{DCEmbed} avoids these additional binaries, finding optimal designs earlier through the sequence of subproblems containing only the 16 original first-stage binaries. 
The \emph{NN-P} architecture is a single layer, and thus \texttt{DCEmbed}'s formulation is equivalent to two-path.

\section{Conclusion}
\label{sec:conclusion}
We presented \texttt{DCEmbed}, a scalable local heuristic for optimizing trained neural surrogates embedded in the objectives, inequalities, and equalities of optimization problems. 
The method combines an exact, reduced-size epigraph formulation of the DC components of the surrogate with penalty CCP, retaining the residual expressions and original discrete decisions without adding activation binaries. 
Across various benchmarks, \texttt{DCEmbed} 
shows stronger progress towards high-quality feasible solutions than exact mixed-integer embeddings.
It achieves over $4\times$ lower NPI in resource allocation and finds the surrogate optimum $\sim5\times$ faster than Gurobi ML in two-stage stochastic programming. 

Avoiding activation binaries shifts the computational bottleneck to the sequence of CCP subproblems, whose cost remains an important limitation at larger scales. 
While our formulation reduces their size for a fixed DC decomposition, a natural next step is to tailor the decomposition itself to the iterative procedure, balancing progress per iteration with subproblem cost to improve scalability and solution quality. 
Furthermore, \texttt{DCEmbed} is a local heuristic that provides no convergence or optimality guarantees, and thus establishing such guarantees remains a valuable theoretical direction. 

\section{AI use statement}
In this work, we used generative AI coding tools to assist with the implementation of our method and baselines. Other tasks requiring disclosure were either performed without generative AI tools or were not applicable to this work. Additionally, we used generative AI tools to help identify relevant literature and edit the research paper to improve readability. We have reviewed all AI-assisted work. LLM-generated code was verified and tested for correctness. We take responsibility for the final content of this work, including text, claims or artifacts produced with the aid of generative AI.

\bibliography{references}
\bibliographystyle{iclr2027_conference}

\appendix
\section{Proofs}

\subsection{Proof of Theorem~\ref{thm:dc_decomposition}}
\label{appendix:dc_decomp_proofs}

\dcdecomposition*

\begin{proof}
We proceed by induction on the layer index $\ell$. We show that each activation
$z^\ell$ can be written as
\begin{equation*}
    z^\ell(x)=G^\ell(x)-H^\ell(x),
\end{equation*}
where every coordinate of $G^\ell$ and $H^\ell$ is convex.

For the first layer, define
\begin{equation*}
    G^1(x)= \operatorname{ReLU}(W^1x+b^1),
    \qquad
    H^1(x)= 0.
\end{equation*}
We can clearly see that $z^1(x)$ admits a DC decomposition as
\begin{equation*}
    z^1(x)=\operatorname{ReLU}(W^1x+b^1)
          =G^1(x)-H^1(x).
\end{equation*}

We now assume that
\begin{equation*}
    z^{\ell-1}(x)
    =
    G^{\ell-1}(x)-H^{\ell-1}(x),
\end{equation*}
where $G^{\ell-1}$ and $H^{\ell-1}$ are componentwise convex. We decompose
$W^\ell$ entrywise as
\begin{equation*}
    W^\ell=W^\ell_+-W^\ell_-,
    \qquad
    W^\ell_+,W^\ell_-\geq 0,
\end{equation*}
and define
\begin{equation*}
\begin{aligned}
    P^\ell(x)
        &=
        W^\ell_+G^{\ell-1}(x)
        +W^\ell_-H^{\ell-1}(x)
        +b^\ell,\\
    Q^\ell(x)
        &=
        W^\ell_-G^{\ell-1}(x)
        +W^\ell_+H^{\ell-1}(x).
\end{aligned}
\end{equation*}

The preactivation $a^\ell$ can be written as

\begin{equation*}
\begin{aligned}
    a^\ell(x)
        &=W^\ell z^{\ell-1}(x)+b^\ell\\
        &=P^\ell(x)-Q^\ell(x).
\end{aligned}
\end{equation*}
We can use the identity
$\operatorname{ReLU}(p-q)=\max\{p,q\}-q$ to obtain
\begin{equation*}
\begin{aligned}
    z^\ell(x)
        &=\operatorname{ReLU}\!\left(P^\ell(x)-Q^\ell(x)\right)\\
        &=\max\!\left\{P^\ell(x),Q^\ell(x)\right\}-Q^\ell(x).
\end{aligned}
\end{equation*}
Therefore, setting
\begin{equation*}
    G^\ell(x)
        =
        \max\!\left\{P^\ell(x),Q^\ell(x)\right\},
    \qquad
    H^\ell(x)
        = Q^\ell(x),
\end{equation*}
yields the desired DC representation.
\end{proof}

\subsection{Proof of Theorem~\ref{thm:extended_representation}}
\label{app:representation_proof}
\extendedrepresentation*

\begin{proof}
We first prove that every feasible $u\in\mathcal{U}(x)$ satisfies
\begin{equation}
    \label{eq:appendix_proof_statement_dominance}
    \wbar{G}_r(u) \geq G_r(x), \quad \wbar{H}_r(u) \geq H_r(x)
\end{equation}

We define
\begin{equation*}
    \begin{aligned}
        &\wtilde{\eta}^1 = 0 \\
        &\wtilde{\eta}^\ell = W^{\ell,-}u^{\ell-1} + |W^\ell|\wtilde{\eta}^{\ell-1}, \quad \ell=2,\dots,L
    \end{aligned}
\end{equation*}

and 
\begin{equation*}
    \wtilde{G}^{\ell}(x)=u^{\ell} + \wtilde{\eta}^{\ell}(x), \quad 
    \wtilde{H}^{\ell}(x)= \wtilde{\eta}^{\ell}(x).
\end{equation*}

To begin, we will show by induction that 
\begin{equation}
    \label{eq:proof_induction_statement}
    \begin{aligned}
        & \wtilde{G}^{\ell}(x) \geq G^{\ell}(x) \\
        & \wtilde{H}^{\ell}(x) \geq H^{\ell}(x)
    \end{aligned}
\end{equation}
For the first layer, we take 

\begin{equation*}
    \begin{aligned}
        & \wtilde{G}^1(x) = u^1 \geq \ReLU (W^1x+b^1) = G^1(x) \\
        & \wtilde{H}^1(x) = 0 = H^1(x).
    \end{aligned}
\end{equation*}

Now, we assume that $\wtilde{G}^{\ell-1}(x) \geq G^{\ell-1}(x)$ and $\wtilde{H}^{\ell-1}(x) \geq H^{\ell-1}(x)$.

We can define
\begin{equation*}
    \begin{aligned}
        & \wtilde{P}^{\ell}(x) = W^{\ell,+}\wtilde{G}^{\ell-1}(x) + W^{\ell,-}\wtilde{H}^{\ell-1}(x) + b^\ell \\
        & \wtilde{Q}^{\ell}(x) = W^{\ell,-}\wtilde{G}^{\ell-1}(x) + W^{\ell,+}\wtilde{H}^{\ell-1}(x).
    \end{aligned}
\end{equation*}

Note that $\wtilde{P}^{\ell}(x) \geq P^\ell(x), \wtilde{Q}^{\ell}(x) \geq Q^\ell(x)$, and that $\wtilde{Q}^{\ell}(x) = \wtilde{\eta}^{\ell}(x) = \wtilde{H}^{\ell}(x)$. Thus, it follows that $\wtilde{H}^{\ell}(x) = \wtilde{Q}^{\ell}(x) \geq Q^{\ell}(x) = H^{\ell}(x)$.

Because $u^\ell \geq 0$, we have 
\begin{equation*}
    \wtilde{G}^{\ell}(x) = u^{\ell} + \wtilde{\eta}^{\ell}(x) \geq \wtilde{Q}^{\ell}(x).
\end{equation*}

Because $u^\ell \geq W^{\ell}u^{\ell-1} + b^\ell$, we have
\begin{equation*}
    \begin{aligned}
        \wtilde{G}^{\ell}(x) &\geq W^\ell u^{\ell-1} + \wtilde{\eta}^{\ell} + b^\ell \\
        &= \wtilde{P}^\ell (x)
    \end{aligned}
\end{equation*}

Thus, we have 
\begin{equation*}
    \begin{aligned}
        \wtilde{G}^{\ell}(x) &\geq \max\left(\wtilde{P}^\ell(x), \wtilde{Q}^\ell(x)\right) \\
    \implies  &\wtilde{G}^{\ell}(x) \geq G^\ell(x).
    \end{aligned}
\end{equation*}

This completes the induction, proving (\ref{eq:proof_induction_statement}). Now, recall that 
\begin{equation*}
    \wbar{G}_r(u) = 
\beta_r+(\nu_r^+)^\top u^L
+\sum_{\ell=1}^{L-1}c_{r,\ell}^\top u^\ell
\end{equation*}

We can rewrite $\wbar{G}_r(u)$ and $G_r(u)$ as 
\begin{equation*}
    \begin{aligned}
        & \wbar{G}_r(u) = \beta_r + (\nu_r^+)^\top \wtilde{G}^L + (\nu_r^-)^\top\wtilde{H}^L
    \end{aligned}
\end{equation*}
and 
\begin{equation*}
    \begin{aligned}
        & {G}_r(x) = \beta_r + (\nu_r^+)^\top {G}^L + (\nu_r^-)^\top{H}^L.
    \end{aligned}
\end{equation*}

Combined with the result of the induction, we can clearly see that $\wbar{G}_r(u) \geq G_r(x)$. A similar argument shows that $\wbar{H}_r(u) \geq H(x)$. 

The above argument has proven (\ref{eq:appendix_proof_statement_dominance}). We now show that there exists a feasible $u$ that attains this bound. Consider the choice
\begin{equation*}
    u^\ell = z^\ell(x) = \ReLU \left(W^\ell z^{\ell-1}(x) + b^\ell \right).
\end{equation*}
This is clearly feasible. Thus, with this assignment, we can see that $\wbar{G}_r(u)$ attains the same value as $G_r(x)$, and $\wbar{H}_r(u)$ attains the same value as $H_r(x)$.
\end{proof}

\section{Baseline Methods}
\label{app:baseline_methods}

\subsection{Big-$M$ Formulation}
\label{app:big_m}

For a $\ReLU$ neural network, consider the preactivations $a^\ell(x) \in \mathbb{R}^{n_\ell}$ of layer $\ell$. Let $a_j^{\ell}(x)$ denote the $j$-th component (neuron) of the preactivation $a^\ell(x)$, with activation $z_j^\ell(x) = \max\{a^\ell_j(x), 0\}$. 

Assuming we have valid bounds on the preactivations given by
\begin{equation*}
    L_j^\ell \leq a_j^\ell(x) \leq U_j^\ell,
\end{equation*}
we can represent the graph of the $\ReLU$ exactly by
\begin{equation*}
    \begin{aligned}
        & z_j^\ell(x) \geq a^\ell_j(x), \quad z_j^\ell(x) \leq a^\ell_j(x) -L^\ell_j\left(1- \delta_j^\ell \right) \\
        & z_j^\ell(x) \geq0, \quad z_j^\ell(x) \leq U_j^\ell \delta_j^\ell \\
        & \delta_j^\ell \in \{0,1\}.
    \end{aligned}
\end{equation*}

\subsection{Big-$M$ + Optimization Based Bound Tightening (OBBT)}
OBBT strengthens the formulation in Appendix~\ref{app:big_m} by tightening the preactivation bounds before optimization \citep{grimstad_relu_2019}. 
Starting from valid bounds, we process the hidden layers in order.
For layer $\ell$, let $\mathcal{F}_{\ell-1}$ denote the feasible set defined by the input constraints and the constraints representing earlier layers, using their current bounds. 
With $z^0=x$, we solve the following auxiliary problems:
\begin{equation}
\begin{aligned}
\widehat L_j^\ell
&\leq\min_{(x,z)\in\mathcal F_{\ell-1}}
\left(W_{j:}^\ell z^{\ell-1}+b_j^\ell\right),\\
\widehat U_j^\ell
&\geq\max_{(x,z)\in\mathcal F_{\ell-1}}
\left(W_{j:}^\ell z^{\ell-1}+b_j^\ell\right),
\end{aligned}
\end{equation}
and we update the preactivation bounds as:
\begin{equation}
L_j^\ell\leftarrow\max\{L_j^\ell,\widehat L_j^\ell\},
\qquad
U_j^\ell\leftarrow\min\{U_j^\ell,\widehat U_j^\ell\}.
\end{equation}
The resulting activation bounds are carried into the auxiliary problems for subsequent layers. 
A neuron requires no activation binary when $U_j^\ell \leq 0$ or $L_j^\ell \geq 0$, since its output is then fixed to zero or its preactivation, respectively. 
Our implementation uses a single complete forward pass, with a time-limited solve for each minimization and maximization problem. 
Even with a 0.1 second limit per solve, the preprocessing time for OBBT can be substantial. 
For the water treatment instance with $800$ samples, a complete forward pass requires $153,600$ auxiliary solves, taking up to $256$ minutes (more than $85\times$ the evaluation budget). 
\subsection{Projected Gradient Descent}
For the continuous QP benchmark, PGD projects onto the convex residual feasible set $\mathcal{X}$ and penalizes the neural inequalities as
\begin{equation}
\begin{aligned}
P_\rho(x)&=\frac12\|Gx\|_2^2+c^\top x
+\frac{\rho}{2}\sum_{r=1}^{30}[f_\theta(x^{(r)})]_+^2,\\
x^{k+1}&=\Pi_{\mathcal X}(x^k-\alpha\nabla P_\rho(x^k)).
\end{aligned}
\end{equation}
We fix $\alpha=0.1$ and $\rho=100$ across architectures and trials. 
We formulate and solve the projection step in Gurobi as a QP. Between consecutive neural feasible and neural infeasible iterates, we search along the line segment between them for a better feasible solution. 

\section{Experiment Details}
\label{app:exp_details}
Within each instance, all methods use the same fixed neural surrogate, residual model, and initial feasible incumbent. 
Solvers run single-threaded, and performance is evaluated over a $180s$ horizon. 
The final exact baseline solves use \texttt{MIPFocus=1} in Gurobi to prioritize finding feasible incumbents. 

We use the same OBBT protocol across all benchmarks, sizes, and trials. Bound-tightening of the big-$M$ formulation precedes the $180s$ optimization horizon, and its time cost is not counted towards the solution time.
This grants the tightened baseline additional preprocessing time followed by its full optimization budget. 

For \texttt{DCEmbed}, we use the penalty update and objective stopping rule in Section~\ref{sec:method} with $\mu=10$, $\tau_{\mathrm{max}}=10^6\tau_0$, and $\epsilon_{\mathrm{obj}}=10^{-6}$. 
The neural feasibility tolerance is $10^{-5}$ for the QP and water experiments. 
Neur2SP embeds only a neural objective. 
Initial penalties are $\tau_0=100$ for the QPs, and $\tau_0=10$ for the water experiments. 
These settings are fixed across problem sizes and trials within each benchmark. 

\subsection{Neural Surrogate Scaling in Random QPs}
\label{app:exp_details_surrogate_scaling}
The residual model has 150 variables $x \in [-1,1]^{150}$, 18 linear equalities, and 150 linear inequalities. 
We use $G\in\mathbb R^{132\times150}$, $E\in\mathbb R^{18\times150}$, and $C\in\mathbb R^{75\times150}$, sampling their entries and the linear objective coefficients independently as
\begin{equation}
G_{ij}\sim\mathcal N\!\left(0,\frac{1}{132}\right),
\qquad
E_{ij},C_{ij}\sim\mathcal N\!\left(0,\frac{1}{150}\right),
\qquad
c_i\sim\mathcal N\!\left(0,\frac{1}{150^2}\right).
\end{equation}
Each trial fixes the residual model across architectures, and all methods start from the feasible point $x^0=0$.

Each surrogate approximates the same nonconvex function on $[-1,1]^{5}$:
\begin{equation}
\begin{aligned}
q(u)={}&0.30u_1^2+0.20u_2^2+0.15u_3^2
+0.15u_4^2+0.10u_5^2\\
&+0.10u_1u_4+0.08u_2u_5
+0.08\bigl(1-\cos(\pi u_3)\bigr)-0.15.
\end{aligned}
\end{equation}
We train the neural surrogate with varying architectures of different depths and widths.
The surrogate is trained by minimizing mean-squared error using $20,000$ iterations of Adam and a minibatch size of $4,000$.
The training data is uniformly sampled from $[-1,1]^5$.
The learning rate decays from $2\times 10^{-3}$ to $10^{-5}$, and the checkpoint with the lowest validation error is retained. 
The three trials vary residual model and training seeds, and within each trial, architectures share the same training data.

\subsection{Scenario Scaling in Two-Stage Stochastic Programming}
\label{app:exp_details_neur2sp}
We use the pooling problem and released \emph{NN-P} architecture from \citet{dumouchelle_neur2sp_2022}, to which we refer the reader for the original problem and training details. 
No additional training is performed on the \emph{NN-P} model. 
We follow the benchmark's scenario discretization with $m\in\{12,\ldots,17\}$ points for each of its three uncertain quantities, giving $S=m^3$ scenarios. 
The original profit objective is negated to obtain the cost minimization convention used here. 

To provide the same initialization to every method, we construct a pool of eight feasible first-stage decisions. 
To construct the pool, we optimize the first-stage linear objective and a randomly sampled Gaussian linear objective over the original constraints. 
We evaluate the surrogate objective at these designs and retain the best five. 
\texttt{DCEmbed} uses them as sequential starts, while the exact formulations receive the same designs as MIP starts. 
Thus all methods receive identical candidate designs and screening information for initialization. 
We enumerate all the feasible first-stage decisions offline to establish the global optimum of the surrogate problem for evaluation. 

\subsection{Resource Allocation for Water Treatment}
\label{app:exp_details_resource_allocation}
The Water Potability dataset pairs nine water quality measurements (pH, hardness, dissolved solids, chloramines, sulfate, conductivity, organic carbon, trihalomethanes, and turbidity) with a binary potability label.
We train a classifier with four hidden $\ReLU$ layers of width 32 on $2,011$ training points. 
Each feature is normalized by subtracting its mean and dividing by its standard deviation. 
Training uses binary cross-entropy with logits and Adam for 80 epochs, with a batch size of 128, learning rate of $0.005$, and weight decay of $0.001$.
The network and normalization are then fixed for all experiments. 

The network outputs a logit, for which zero corresponds to a predicted potability probability of $0.5$.
We count a treated sample only when its score reaches the small positive margin $\epsilon=10^{-4}$.
Candidate water samples are records labeled non-potable that do not already satisfy this score requirement before treatment. 
For each of the sampling seeds, we shuffle the eligible records once and use the first $N$ as the candidate set. 
Addition and removal budgets follow the SurrogateLIB generator and remain fixed as $N$ increases.
All methods start from the untreated configuration with no samples counted as potable. 

\section{Additional Results}
\label{app:additional results}

\subsection{Solution Trajectories}
NPI combines solution quality and time into a single metric. To examine these effects separately, we show feasible incumbent trajectories in Figure~\ref{fig:appendix_trajectories} at the largest tested sizes for the experiments in Section~\ref{sec:experiments}.
\begin{figure}[t]
    \centering
    \includegraphics[width=\linewidth]{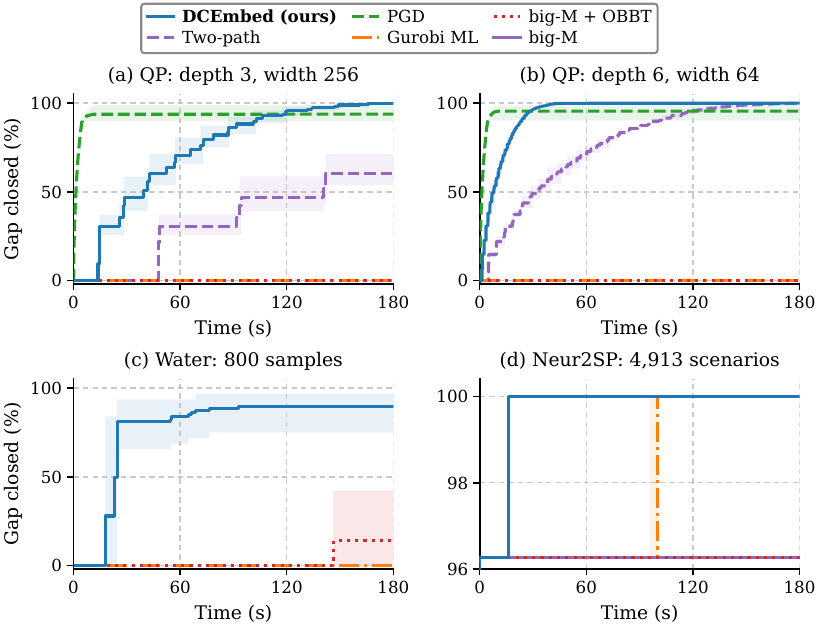}
    \caption{Primal solution trajectories showing gap closed over time (higher is better).}
    \label{fig:appendix_trajectories}
\end{figure}
For the QP experiment, PGD improves rapidly before stalling, while \texttt{DCEmbed} continues to improve and reaches better final objectives. 
\texttt{DCEmbed} also makes substantially faster progress than the two-path formulation, while both exact embeddings remain at the initial incumbent. 
In water treatment, \texttt{DCEmbed} finds strong incumbents early, whereas the exact embeddings make little or no progress. 
For Neur2SP, \texttt{DCEmbed} reaches the same optimum as Gurobi ML substantially earlier, while both big-$M$ variants stall at suboptimal designs. 

\subsection{Final Solution Quality}
Tables~\ref{tab:qp-final-objectives}--\ref{tab:neur2sp-final-objectives} report the best feasible objectives reached within the evaluation budget. 
The best objectives, including ties, are \textbf{bolded} in the tables below. 

\begin{table}[htbp]
\centering
\caption{Synthetic QP objectives (lower is better). Values are averaged over three trials, with zero indicating the initial incumbent objective.}
\label{tab:qp-final-objectives}
\vspace{4pt}
\small
\setlength{\tabcolsep}{3pt}
\renewcommand{\arraystretch}{1.35}
\begin{tabular*}{\linewidth}{@{\extracolsep{\fill}}rrccccc@{}}
\toprule[0.8pt]
\textbf{Depth} & \textbf{Width} & \textbf{DCEmbed (ours)} & \textbf{Two-path} & \textbf{PGD} & \textbf{Gurobi ML} & \textbf{big-M + OBBT} \\
\midrule[0.4pt]
1 & 64 & $\boldsymbol{-0.06527}$ & $\boldsymbol{-0.06527}$ & $-0.05990$ & $\boldsymbol{-0.06527}$ & $\boldsymbol{-0.06527}$ \\
2 & 64 & $\boldsymbol{-0.06522}$ & $\boldsymbol{-0.06522}$ & $-0.06014$ & $0.00000$ & $-0.00216$ \\
3 & 8 & $\boldsymbol{-0.06531}$ & $-0.06529$ & $-0.06315$ & $\boldsymbol{-0.06531}$ & $\boldsymbol{-0.06531}$ \\
3 & 16 & $\boldsymbol{-0.06532}$ & $\boldsymbol{-0.06532}$ & $-0.06203$ & $-0.03866$ & $\boldsymbol{-0.06532}$ \\
3 & 32 & $\boldsymbol{-0.06549}$ & $-0.06547$ & $-0.06267$ & $0.00000$ & $0.00000$ \\
3 & 64 & $-0.06526$ & $\boldsymbol{-0.06527}$ & $-0.06036$ & $0.00000$ & $0.00000$ \\
3 & 128 & $-0.06520$ & $\boldsymbol{-0.06522}$ & $-0.05996$ & $0.00000$ & $0.00000$ \\
3 & 256 & $\boldsymbol{-0.06487}$ & $-0.03864$ & $-0.06059$ & $0.00000$ & $0.00000$ \\
4 & 64 & $\boldsymbol{-0.06535}$ & $-0.06504$ & $-0.06007$ & $0.00000$ & $0.00000$ \\
5 & 64 & $\boldsymbol{-0.06522}$ & $-0.06455$ & $-0.06162$ & $0.00000$ & $0.00000$ \\
6 & 64 & $\boldsymbol{-0.06494}$ & $-0.06481$ & $-0.06174$ & $0.00000$ & $0.00000$ \\
\bottomrule[0.8pt]
\end{tabular*}
\end{table}

\begin{table}[htbp]
\centering
\caption{Number of potable water samples (higher is better), averaged over three trials.}
\label{tab:water-final-objectives}
\vspace{4pt}
\small
\setlength{\tabcolsep}{3pt}
\renewcommand{\arraystretch}{1.35}
\begin{tabular*}{\linewidth}{@{\extracolsep{\fill}}rccc@{}}
\toprule[0.8pt]
\textbf{Samples} & \textbf{DCEmbed (ours)} & \textbf{Gurobi ML} & \textbf{big-M + OBBT} \\
\midrule[0.4pt]
25 & $\boldsymbol{10.00}$ & $\boldsymbol{10.00}$ & $9.33$ \\
50 & $\boldsymbol{20.33}$ & $15.33$ & $15.00$ \\
100 & $\boldsymbol{26.00}$ & $17.67$ & $13.33$ \\
200 & $\boldsymbol{26.33}$ & $8.33$ & $12.67$ \\
400 & $\boldsymbol{30.33}$ & $1.33$ & $8.00$ \\
800 & $\boldsymbol{31.33}$ & $0.00$ & $4.33$ \\
\bottomrule[0.8pt]
\end{tabular*}
\end{table}

\begin{table}[htbp]
\centering
\caption{Neur2SP objective (higher is better), averaged over three trials.}
\label{tab:neur2sp-final-objectives}
\vspace{4pt}
\small
\setlength{\tabcolsep}{3pt}
\renewcommand{\arraystretch}{1.35}
\begin{tabular*}{\linewidth}{@{\extracolsep{\fill}}rcccc@{}}
\toprule[0.8pt]
\textbf{Scenarios} & \textbf{DCEmbed (ours)} & \textbf{Gurobi ML} & \textbf{big-M} & \textbf{big-M + OBBT} \\
\midrule[0.4pt]
1,728 & $\boldsymbol{171.5118}$ & $\boldsymbol{171.5118}$ & $\boldsymbol{171.5118}$ & $\boldsymbol{171.5118}$ \\
2,197 & $\boldsymbol{169.7280}$ & $\boldsymbol{169.7280}$ & $\boldsymbol{169.7280}$ & $\boldsymbol{169.7280}$ \\
2,744 & $\boldsymbol{169.9166}$ & $\boldsymbol{169.9166}$ & $\boldsymbol{169.9166}$ & $161.8368$ \\
3,375 & $\boldsymbol{169.9523}$ & $\boldsymbol{169.9523}$ & $161.5369$ & $\boldsymbol{169.9523}$ \\
4,096 & $\boldsymbol{169.1168}$ & $\boldsymbol{169.1168}$ & $161.4644$ & $161.4644$ \\
4,913 & $\boldsymbol{169.4088}$ & $\boldsymbol{169.4088}$ & $161.2606$ & $161.2606$ \\
\bottomrule[0.8pt]
\end{tabular*}
\end{table}

\end{document}

%% file: math_commands.tex
\usepackage{amsmath,amsfonts,bm}

\def\eqref#1{equation~\ref{#1}}

\def\1{\bm{1}}

\DeclareMathAlphabet{\mathsfit}{\encodingdefault}{\sfdefault}{m}{sl}
\SetMathAlphabet{\mathsfit}{bold}{\encodingdefault}{\sfdefault}{bx}{n}



%% file: references.bib
@book{goodfellow_deep_2016,
	title = {Deep {Learning}},
	publisher = {MIT Press},
	author = {Goodfellow, Ian and Bengio, Yoshua and Courville, Aaron},
	year = {2016},
}

@inproceedings{gasse_machine_2022,
	series = {Proceedings of {Machine} {Learning} {Research}},
	title = {The {Machine} {Learning} for {Combinatorial} {Optimization} {Competition} ({ML4CO}): {Results} and {Insights}},
	volume = {176},
	url = {https://proceedings.mlr.press/v176/gasse22a.html},
	booktitle = {Proceedings of the {NeurIPS} 2021 {Competitions} and {Demonstrations} {Track}},
	publisher = {PMLR},
	author = {Gasse, Maxime and Bowly, Simon and Cappart, Quentin and Charfreitag, Jonas and Charlin, Laurent and Chételat, Didier and Chmiela, Antonia and Dumouchelle, Justin and Gleixner, Ambros and Kazachkov, Aleksandr M. and Khalil, Elias and Lichocki, Pawel and Lodi, Andrea and Lubin, Miles and Maddison, Chris J. and Christopher, Morris and Papageorgiou, Dimitri J. and Parjadis, Augustin and Pokutta, Sebastian and Prouvost, Antoine and Scavuzzo, Lara and Zarpellon, Giulia and Yang, Linxin and Lai, Sha and Wang, Akang and Luo, Xiaodong and Zhou, Xiang and Huang, Haohan and Shao, Shengcheng and Zhu, Yuanming and Zhang, Dong and Quan, Tao and Cao, Zixuan and Xu, Yang and Huang, Zhewei and Zhou, Shuchang and Binbin, Chen and Minggui, He and Hao, Hao and Zhiyu, Zhang and Zhiwu, An and Kun, Mao},
	editor = {Kiela, Douwe and Ciccone, Marco and Caputo, Barbara},
	month = dec,
	year = {2022},
	pages = {220--231},
}

@article{berthold_measuring_2013,
	title = {Measuring the impact of primal heuristics},
	volume = {41},
	issn = {0167-6377},
	url = {https://www.sciencedirect.com/science/article/pii/S0167637713001181},
	doi = {https://doi.org/10.1016/j.orl.2013.08.007},
	number = {6},
	journal = {Operations Research Letters},
	author = {Berthold, Timo},
	year = {2013},
	pages = {611--614},
}

@misc{kadiwal_water_2021,
	title = {Water {Quality}: {Drinking} {Water} {Potability}},
	url = {https://www.kaggle.com/datasets/adityakadiwal/water-potability},
	publisher = {Kaggle},
	author = {Kadiwal, Aditya},
	year = {2021},
}

@incollection{paszke_pytorch_2019,
	address = {Red Hook, NY, USA},
	title = {{PyTorch}: an imperative style, high-performance deep learning library},
	booktitle = {Proceedings of the 33rd {International} {Conference} on {Neural} {Information} {Processing} {Systems}},
	publisher = {Curran Associates Inc.},
	author = {Paszke, Adam and Gross, Sam and Massa, Francisco and Lerer, Adam and Bradbury, James and Chanan, Gregory and Killeen, Trevor and Lin, Zeming and Gimelshein, Natalia and Antiga, Luca and Desmaison, Alban and Köpf, Andreas and Yang, Edward and DeVito, Zach and Raison, Martin and Tejani, Alykhan and Chilamkurthy, Sasank and Steiner, Benoit and Fang, Lu and Bai, Junjie and Chintala, Soumith},
	year = {2019},
}

@misc{gurobi_optimization_llc_gurobi_2026,
	title = {Gurobi {Optimizer} {Reference} {Manual}},
	url = {https://www.gurobi.com},
	author = {{Gurobi Optimization, LLC}},
	year = {2026},
}

@book{beck_first-order_2017,
	address = {Philadelphia, PA},
	title = {First-{Order} {Methods} in {Optimization}},
	isbn = {9781611974980 9781611974997},
	url = {http://epubs.siam.org/doi/book/10.1137/1.9781611974997},
	doi = {10.1137/1.9781611974997},
	language = {en},
	urldate = {2026-08-13},
	publisher = {Society for Industrial and Applied Mathematics},
	author = {Beck, Amir},
	month = oct,
	year = {2017},
}

@inproceedings{tsay_partition-based_2021,
	address = {Red Hook, NY, USA},
	series = {{NIPS} '21},
	title = {Partition-based formulations for mixed-integer optimization of trained {ReLU} neural networks},
	isbn = {9781713845393},
	booktitle = {Proceedings of the 35th {International} {Conference} on {Neural} {Information} {Processing} {Systems}},
	publisher = {Curran Associates Inc.},
	author = {Tsay, Calvin and Kronqvist, Jan and Thebelt, Alexander and Misener, Ruth},
	year = {2021},
}

@misc{gurobi_optimization_llc_gurobi_nodate,
	title = {Gurobi {Machine} {Learning}},
	url = {https://gurobi-machinelearning.readthedocs.io/},
	author = {{Gurobi Optimization, LLC}},
	note = {Published: Software documentation},
}

@misc{turner_pyscipopt-ml_2024,
	title = {{PySCIPOpt}-{ML}: {Embedding} {Trained} {Machine} {Learning} {Models} into {Mixed}-{Integer} {Programs}},
	shorttitle = {{PySCIPOpt}-{ML}},
	url = {http://arxiv.org/abs/2312.08074},
	doi = {10.48550/arXiv.2312.08074},
	urldate = {2026-08-07},
	publisher = {arXiv},
	author = {Turner, Mark and Chmiela, Antonia and Koch, Thorsten and Winkler, Michael},
	month = may,
	year = {2024},
	note = {arXiv:2312.08074 [math.OC]},
}

@article{ceccon_omlt_2022,
	title = {{OMLT}: {Optimization} \& {Machine} {Learning} {Toolkit}},
	volume = {23},
	number = {349},
	journal = {Journal of Machine Learning Research},
	author = {Ceccon, F. and Jalving, J. and Haddad, J. and Thebelt, A. and Tsay, C. and Laird, C. D and Misener, R.},
	year = {2022},
	pages = {1--8},
}

@misc{li_bridging_2026,
	title = {Bridging {Control} with {Neural} {Network} {Verifier} alpha-beta-{CROWN}: {A} {Tutorial}},
	shorttitle = {Bridging {Control} with {Neural} {Network} {Verifier} alpha-beta-{CROWN}},
	url = {http://arxiv.org/abs/2605.26577},
	doi = {10.48550/arXiv.2605.26577},
	urldate = {2026-08-06},
	publisher = {arXiv},
	author = {Li, Haoyu and Zhong, Xiangru and Cheng, Hao and Hu, Bin and Zhang, Huan},
	month = may,
	year = {2026},
	note = {arXiv:2605.26577 [eess.SY]},
}

@article{awasthi_dc-programming_2024,
	title = {{DC}-programming for neural network optimizations},
	volume = {95},
	issn = {0925-5001, 1573-2916},
	url = {https://link.springer.com/10.1007/s10898-023-01344-2},
	doi = {10.1007/s10898-023-01344-2},
	language = {en},
	number = {1},
	urldate = {2026-07-21},
	journal = {Journal of Global Optimization},
	author = {Awasthi, Pranjal and Mao, Anqi and Mohri, Mehryar and Zhong, Yutao},
	month = jan,
	year = {2024},
	pages = {5--21},
}

@inproceedings{zhang_branch_2022,
	series = {Proceedings of {Machine} {Learning} {Research}},
	title = {A {Branch} and {Bound} {Framework} for {Stronger} {Adversarial} {Attacks} of {ReLU} {Networks}},
	volume = {162},
	url = {https://proceedings.mlr.press/v162/zhang22ae.html},
	booktitle = {Proceedings of the 39th {International} {Conference} on {Machine} {Learning}},
	publisher = {PMLR},
	author = {Zhang, Huan and Wang, Shiqi and Xu, Kaidi and Wang, Yihan and Jana, Suman and Hsieh, Cho-Jui and Kolter, Zico},
	editor = {Chaudhuri, Kamalika and Jegelka, Stefanie and Song, Le and Szepesvari, Csaba and Niu, Gang and Sabato, Sivan},
	month = jul,
	year = {2022},
	pages = {26591--26604},
}

@inproceedings{wang_beta-crown_2021,
	title = {Beta-{CROWN}: {Efficient} {Bound} {Propagation} with {Per}-neuron {Split} {Constraints} for {Neural} {Network} {Robustness} {Verification}},
	volume = {34},
	url = {https://proceedings.neurips.cc/paper_files/paper/2021/file/fac7fead96dafceaf80c1daffeae82a4-Paper.pdf},
	booktitle = {Advances in {Neural} {Information} {Processing} {Systems}},
	publisher = {Curran Associates, Inc.},
	author = {Wang, Shiqi and Zhang, Huan and Xu, Kaidi and Lin, Xue and Jana, Suman and Hsieh, Cho-Jui and Kolter, J. Zico},
	editor = {Ranzato, M. and Beygelzimer, A. and Dauphin, Y. and Liang, P. S. and Vaughan, J. Wortman},
	year = {2021},
	pages = {29909--29921},
}

@article{maragno_mixed-integer_2025,
	title = {Mixed-{Integer} {Optimization} with {Constraint} {Learning}},
	volume = {73},
	issn = {0030-364X, 1526-5463},
	url = {https://pubsonline.informs.org/doi/10.1287/opre.2021.0707},
	doi = {10.1287/opre.2021.0707},
	language = {en},
	number = {2},
	urldate = {2026-07-22},
	journal = {Operations Research},
	author = {Maragno, Donato and Wiberg, Holly and Bertsimas, Dimitris and Birbil, Ş. İlker and Den Hertog, Dick and Fajemisin, Adejuyigbe O.},
	month = mar,
	year = {2025},
	pages = {1011--1028},
}

@article{fajemisin_optimization_2024,
	title = {Optimization with constraint learning: {A} framework and survey},
	volume = {314},
	issn = {03772217},
	shorttitle = {Optimization with constraint learning},
	url = {https://linkinghub.elsevier.com/retrieve/pii/S0377221723003405},
	doi = {10.1016/j.ejor.2023.04.041},
	language = {en},
	number = {1},
	urldate = {2026-07-22},
	journal = {European Journal of Operational Research},
	author = {Fajemisin, Adejuyigbe O. and Maragno, Donato and Den Hertog, Dick},
	month = apr,
	year = {2024},
	pages = {1--14},
}

@article{anderson_strong_2020,
	title = {Strong mixed-integer programming formulations for trained neural networks},
	volume = {183},
	issn = {1436-4646},
	url = {https://doi.org/10.1007/s10107-020-01474-5},
	doi = {10.1007/s10107-020-01474-5},
	language = {en},
	number = {1},
	urldate = {2026-07-22},
	journal = {Mathematical Programming},
	author = {Anderson, Ross and Huchette, Joey and Ma, Will and Tjandraatmadja, Christian and Vielma, Juan Pablo},
	month = sep,
	year = {2020},
	pages = {3--39},
}

@article{grimstad_relu_2019,
	title = {{ReLU} networks as surrogate models in mixed-integer linear programs},
	volume = {131},
	issn = {00981354},
	url = {https://linkinghub.elsevier.com/retrieve/pii/S0098135419307203},
	doi = {10.1016/j.compchemeng.2019.106580},
	language = {en},
	urldate = {2026-07-22},
	journal = {Computers \& Chemical Engineering},
	author = {Grimstad, Bjarne and Andersson, Henrik},
	month = dec,
	year = {2019},
	pages = {106580},
}

@article{fischetti_deep_2018,
	title = {Deep neural networks and mixed integer linear optimization},
	volume = {23},
	issn = {1383-7133, 1572-9354},
	url = {http://link.springer.com/10.1007/s10601-018-9285-6},
	doi = {10.1007/s10601-018-9285-6},
	language = {en},
	number = {3},
	urldate = {2026-07-22},
	journal = {Constraints},
	author = {Fischetti, Matteo and Jo, Jason},
	month = jul,
	year = {2018},
	pages = {296--309},
}

@article{liu_optimization_2025,
	title = {Optimization {Over} {Trained} {Neural} {Networks}: {Difference}-of-{Convex} {Algorithm} and {Application} to {Data} {Center} {Scheduling}},
	volume = {9},
	issn = {2475-1456},
	shorttitle = {Optimization {Over} {Trained} {Neural} {Networks}},
	url = {https://ieeexplore.ieee.org/document/11028923},
	doi = {10.1109/LCSYS.2025.3577571},
	urldate = {2026-07-22},
	journal = {IEEE Control Systems Letters},
	author = {Liu, Xinwei and Dvorkin, Vladimir},
	year = {2025},
	pages = {835--840},
}

@inproceedings{sosnin_scaling_2024,
	title = {Scaling {Mixed}-{Integer} {Programming} for {Certification} of {Neural} {Network} {Controllers} {Using} {Bounds} {Tightening}},
	url = {https://ieeexplore.ieee.org/document/10886078},
	doi = {10.1109/CDC56724.2024.10886078},
	urldate = {2026-07-21},
	booktitle = {2024 {IEEE} 63rd {Conference} on {Decision} and {Control} ({CDC})},
	author = {Sosnin, Philip and Tsay, Calvin},
	month = dec,
	year = {2024},
	note = {ISSN: 2576-2370},
	pages = {1645--1650},
}

@article{lipp_variations_2016,
	title = {Variations and extension of the convex–concave procedure},
	volume = {17},
	issn = {1389-4420, 1573-2924},
	url = {http://link.springer.com/10.1007/s11081-015-9294-x},
	doi = {10.1007/s11081-015-9294-x},
	language = {en},
	number = {2},
	urldate = {2026-07-21},
	journal = {Optimization and Engineering},
	author = {Lipp, Thomas and Boyd, Stephen},
	month = jun,
	year = {2016},
	pages = {263--287},
}

@inproceedings{dumouchelle_neur2sp_2022,
	address = {Red Hook, NY, USA},
	series = {{NIPS} '22},
	title = {{Neur2SP}: neural two-stage stochastic programming},
	isbn = {9781713871088},
	booktitle = {Proceedings of the 36th {International} {Conference} on {Neural} {Information} {Processing} {Systems}},
	publisher = {Curran Associates Inc.},
	author = {Dumouchelle, Justin and Patel, Rahul and Khalil, Elias B. and Bodur, Merve},
	year = {2022},
	note = {event-place: New Orleans, LA, USA},
}

@inproceedings{chen_compact_2024,
	series = {{ICML}'24},
	title = {Compact optimality verification for optimization proxies},
	booktitle = {Proceedings of the 41st {International} {Conference} on {Machine} {Learning}},
	publisher = {JMLR.org},
	author = {Chen, Wenbo and Zhao, Haoruo and Tanneau, Mathieu and Van Hentenryck, Pascal},
	year = {2024},
	note = {event-place: Vienna, Austria},
}

@article{kody_modeling_2022,
	title = {Modeling the {AC} power flow equations with optimally compact neural networks: {Application} to unit commitment},
	volume = {213},
	issn = {03787796},
	shorttitle = {Modeling the {AC} power flow equations with optimally compact neural networks},
	url = {https://linkinghub.elsevier.com/retrieve/pii/S0378779622004771},
	doi = {10.1016/j.epsr.2022.108282},
	language = {en},
	urldate = {2026-07-21},
	journal = {Electric Power Systems Research},
	author = {Kody, Alyssa and Chevalier, Samuel and Chatzivasileiadis, Spyros and Molzahn, Daniel},
	month = dec,
	year = {2022},
	pages = {108282},
}
